\documentclass[12pt]{amsart}

\newtheorem{theorem}{Theorem}[section]

\newtheorem{lemma}[theorem]{Lemma}

\theoremstyle{definition}

\newtheorem{remark}[theorem]{Remark}

\newcommand{\ZZ}{ \ensuremath{\mathbb{Z}}}
\newcommand{\CC}{ \ensuremath{\mathbb{C}}}
\newcommand{\HH}{ \ensuremath{\mathbb{H}}}
\newcommand{\RR}{ \ensuremath{\mathbb{R}}}

\usepackage{color}

\begin{document}

\title
{Polyhedrons and PBIBDs from hyperbolic manifolds}

\author{Eran Nevo}
\address{
Department of Mathematics,
Ben Gurion University of the Negev and the Hebrew University of Jerusalem, Israel
}
\email{nevo@math.huji.ac.il}


\thanks{
Research partially supported by an NSF Award DMS-0757828 and by an ISF grant 805/11.
}


\maketitle
\begin{center}
\emph{Dedicated to the memory of dear William Thurston}
\end{center}

\begin{abstract}
By taking quotients of a certain tiling of hyperbolic plane / space by certain group actions, we obtain geometric polyhedra / cellulations with interesting symmetries and incidence structure.
\end{abstract}

\section{Introduction}
I was fortunate to know and work with Bill Thurston.
I am very thankful to Bill for this. We met when I was a postdoc at Cornell, and tried to figure out how many $n$-vertex triangulations the $3$-sphere has, see \cite{Pfeifle-Ziegler} for the bounds known at that time\footnote{In the course of publication, with F. Santos and S. Wilson we obtained almost tight bounds for this problem by improving the lower bound to $2^{\Omega(n^{2})}$ \cite{Nevo-Santos-Wilson}.}.
%
We construct $3$-dimensional cellulations $X_q$ made of octahedral cells on which $PSL_2(q)$ acts transitively. On the one hand, modifying the $X_q$ to manifolds yields a construction of $2^{\Theta(n^{3/2})}$ $n$-vertex triangulations of $3$-manifolds (which, by recent \cite{Nevo-Santos-Wilson}, is not so many).
On the other hand, the symmetries of $X_q$ lead to partially balanced incomplete block designs (PBIBDs) which seem to be new.

PBIBDs are defined via association schemes, bridging algebraic graph theory and design theory, with applications in various areas -- e.g. for the design of statistical experiments, for software testing and cryptography. In mathematics, PBIBDs arise from other contexts as well, for example from finite geometries and finite groups.
In Section \ref{sec:PBIBD} we give the relevant definitions; for further background, as well as references to the wide literature on combinatorial designs, see e.g. \cite{Cameron-VanLint, HandbookCombiDesigns, Street-Street}.

Specifically, our construction yields for every $q>5$ which is either a prime equals $1$ modulo $4$ or a square of a prime $p$ with $p$ equals $3$ modulo $4$,
a PBIBD
with blocks of size $6$, $v=\frac{q^2-1}{4}$ vertices, $b=\frac{q(q^2-1)}{24}$ blocks, each vertex is in exactly $q$ blocks,
and while the number of associated classes grows linearly with $q$, each $2$-subset can be only in $\lambda_i=4$ or $1$ or zero blocks.


The construction is based on quotients of hyperbolic $3$-space tiled by octahedra, by certain discrete subgroups of the Picard group. Applying this method to the hyperbolic plane tiled by ideal triangles, we recover polyhedral surfaces constructed by McMullen \cite{McMullen:PLG91} and Wajima \cite{Wajima89}.

Outline: in Section \ref{sec:Construction} we describe the construction, prove its relevant fundamental properties, and use it to obtain triangulated $3$-manifolds, in Section \ref{sec:PBIBD} we review PBIBDs and view our construction as a PBIBD, in Section \ref{sec:Surface} we relate a variation of the construction to polyhedral surfaces.


\section{Construction}\label{sec:Construction}

Our construction is based on discrete subgroups of the Picard group acting on hyperbolic space.
Let $\ZZ[i]$ be the gaussian integers. Then $\Gamma:=PSL_2(\ZZ[i])$ is a discrete subgroup of $PSL_2(\CC)$ which acts on the hyperbolic $3$-space $\HH=\HH^3$ by orientation preserving isometries, given by M\"{o}bius transformations on the boundary $\CC\cup\{\infty\}$.
The quotient space $\HH/\Gamma$ is an orbifold: topologically it is the 3-sphere minus one cusp point.  It is the double of a  simplex having a chain of 3 edges with angles $\pi/4, \pi/4, \pi/3$ and all other angles $\pi/2$, Cf.\cite[Picture $D=-4$]{Hatcher:Bianchi}.

If $G$ is any torsion-free subgroup  of $\Gamma$ with index $[\Gamma:G]=m$, then $\HH/G$ is a $3$-manifold (with cusps) which is an $m$-cover of $\HH/\Gamma$, where each cusp of $\HH/G$ is a torus.
In this case, $G$ intersects trivially the orientation preserving isometry group of the octahedron, denoted by $O$, and hence a tiling of $\HH$ by octahedra, on which $\Gamma$ acts, induces a tiling of $\HH/G$ by octahedra, on which $\Gamma / G$ acts.

The tiling of $\HH$ we have in mind is the following: in the upper half space model, look on the radius $\frac{1}{\sqrt{2}}$ hemispheres with centers at $(\frac{1}{2}+\frac{1}{2}i)+\ZZ[i]$ in the bottom plane $\CC$. The intersections with the fixed hemisphere with center at $\frac{1}{2}+\frac{1}{2}i$ consist of $4$ arcs bounding a square on this hemisphere (its vertical projection on the bottom plane $\CC$ is the unit square with vertices $0,1,i,1+i$), above which is the upper half of an ideal octahedron $T$, and its other half is obtained by inversion of the upper half w.r.t the hemisphere with center at $\frac{1}{2}+\frac{1}{2}i$.
Thus $\frac{1}{2}+\frac{1}{2}i$ and $\infty$ are opposite vertices in this octahedron, denoted by $T$.

Note that $\Gamma\cap O$ has order $12$, and acts transitively on the vertices, and on the edges of $T$ (but not on the (vertex,edge) flags). In fact, using the Euclidean octahedron, the elements of $\Gamma\cap O$ are rotations by multiples of $\pi/3$ degrees around either of the $4$ axes through centers of opposite triangles, and rotations by multiples of $\pi$ (but not $\pi/2$\ !) degrees around either of the $3$ axes through opposite vertices. We omit the computation.

The quotient of $T$ by $\Gamma\cap O$ gives the Ford fundamental domain $\HH/\Gamma$ described earlier.
The tiling of $\HH$ by octahedra we look at consists of the octahedra $\Gamma(T)$.
See \cite{Hatcher:HyperbolicStructure} for more details and pictures.

For any prime ideal $I$ of $\ZZ[i]$ with quotient field $F_q$ of order $q$, let $\Gamma_0(I)$ be the kernel of the surjective map $PSL_2(\ZZ[i] )\to PSL_2(\ZZ[i] / I)$.  If the characteristic of $F_q$ is not $2$, then $\Gamma_0(I)$ is torsion-free.
This is known. As we could not find a reference, here is a proof.
Indeed,
the characteristic polynomial of an element $A\in \Gamma_0(I)$ over $\CC$ is $p_A(x)=x^2-(2+a)x+(1+b)$ where $a,b\in I$.
If $A$ is torsion, $A$ must be diagonalizable over $\CC$, with eigenvalues being two conjugated roots of unity in the set $\{\pm 1,\pm i,\pm \sqrt{i},\pm i\sqrt{i}\}$.
If $F_q$ has characteristic not $2$, then $p_A(x)$ is not in $\{x^2+1, x^2-i, x^2+i\}$, hence $p_A(x)=(x-1)^2$ and
 $A$ must be the identity.

Assume the characteristic of $F_q$ is not 2.

\begin{lemma}\label{lem:count cusps and octahedra}

1. $[\Gamma:\Gamma_0(I)]=\frac{q(q^2-1)}{2}$ (which is divisible by 12). Hence the tiling above induces a tiling of $\HH/\Gamma_0(I)$ by $\frac{q(q^2-1)}{24}$ octahedra.

2. The stabilizer of a cusp in $PSL_2(F_q)$ has size $2q$. Hence there are $\frac{q^2-1}{4}$ cusps in $\HH/\Gamma_0(I)$.

3. Each cusp of $\HH/\Gamma_0(I)$ is a torus tiled by $q$ squares, looking like $\ZZ[i]/I$.
\end{lemma}
\begin{proof}
For part (1) compute $[\Gamma:\Gamma_0(I)]=|PSL_2(F_q)|=\frac{q(q^2-1)}{2}$.
The number of octahedra is $\frac{[\Gamma:\Gamma_0(I)]}{|\Gamma\cap O|}=q(q^2-1)/24$.

For parts (2) and (3), it is enough to consider the cusp ``at infinity" (namely, the one \emph{not} in the bottom plane $\CC$), as $\Gamma$ acts transitively on the edges of the tiling of $\HH$, hence acts transitively on the cusps.

Note that whenever a group $G$ acts on a space $X$ and $H$ is a normal subgroup of $G$, then the action of $G/H$ on $X/H$ satisfies for any $x\in X$ and its image $\bar{x}\in X/G$ that their stabilizers are related by
$$
\rm{Stab}_{G/H}(\bar{x})=\rm{Stab}_{G}(x)H/H
.$$
Clearly $\rm{Stab}_{\Gamma}(\infty)=\{[^{u \ x}_{0 \ u^{-1}}]: u,x\in \ZZ[i], u=1,i\}$.
Hence, the stabilizer in $PSL_2(F_q)$ of the cusp at infinity in $\HH/\Gamma_0(I)$, denote it simply by $\rm{Stab}(\infty)$, consists  of the distinct representative matrices $[^{u \ x}_{0 \ u^{-1}}]$ with $x\in F_q$ and $u\in\{1,i\}$, thus it has size $2q$. By (1), the number of cusps is $(q(q^2-1)/2)/2q=(q^2-1)/4$. The edges at infinity in the tiling of $\HH$ are the lines from the points $\ZZ[i]$ in the bottom plane that are perpendicular to the bottom plane, thus the
torus at infinity in $\HH/\Gamma_0(I)$ has a tiling by $q$ squares and $q$ vertices, isomorphic to $\ZZ[i]/I$.
\end{proof}

Next, we verify that
\begin{theorem}\label{lem:X}
Let $q>5$.
Then the one-point compactification at each cusp yields a strongly regular cellulation.
\end{theorem}
Note that this cellulation gives a pseudomanifold where the vertex links of the added points are tori.
Following \cite{Pfeifle-Ziegler}, recall that a \emph{cellulation} $C$ of a topological space $M$ is a finite CW complex whose underlying space is $M$. $C$ is \emph{regular} if all the closed cells are embedded, and \emph{strongly regular}\footnote{Not to be confused with the similarly named terms in algebraic graph theory.} if in addition the intersection of any two closed cells is a closed cell.

The theorem immediately follows from the following key lemma.

\begin{lemma}\label{lem:combinatorics of tiling}
Let $q>5$. Then the above tiling of $\HH/\Gamma_0(I)$ satisfies:

1. The vertices of each octahedron are in $6$ distinct cusps.

2. No two edges connect the same two cusps.

3. If two octahedra $P_1,P_2$ have pairs of antipodal vertices sharing the same two cusps, then $P_1=P_2$.


4. For any two cusps $a$ and $b$: if there is
an octahedron in the tiling with opposite ends at cusps $a,b$  then no edge of the tiling connects $a$ and $b$.
\end{lemma}
\begin{proof}
(1) As $\Gamma$ acts transitively on the octahedra of the tiling, it is enough to show that the octahedron $T$ has $6$ distinct vertices in $\HH/\Gamma_0(I)$.
As $\Gamma\cap O$ acts transitively on the vertices of $T$ in $\HH$, and includes $\pi$-rotations through antipodal vertices, we only need to check that $\infty$ is in a different cusp then each of the vertices $0,1,\frac{1}{2}+\frac{1}{2}i=\frac{1}{1-i}$.

Let $p$ generate $I$.
Note that for any $\phi\in\Gamma_0(I)$, $\phi(\infty)$ is of the form $\frac{ap+1}{cp}$ for $a,c\in \ZZ[i]$.
Thus, if $\phi(\infty)=1$ then $1=0\pmod{p}$, a contradiction. If
$\phi(\infty)=0$ then again $1=0\pmod{p}$. If $\phi(\infty)=\frac{1}{1-i}$ then $1-i=0\pmod{p}$, but for $q>2$ this is a contradiction.

(2) By the transitive action of $\Gamma/\Gamma_0(I)$ on the cusps, it is enough to show that any two edges in $\HH$ with one end at infinity have the other ends, denoted by $v,z\in \ZZ[i]$ at two different cusps in $\HH/\Gamma_0(I)$.

Let $\phi=[^{0\  -1}_{1\ -v}]$, so $\phi(v)=\infty$ and
$\rm{Stab}_{\Gamma}(v)=\phi^{-1}\rm{Stab}_{\Gamma}(\infty)\phi$
which equals
$\{ \begin{bmatrix}
-bv+d & av+bv^2-dv \\
-b & a+bv
\end{bmatrix}
:\ ad=1,\ a,b,d\in\ZZ[i]\}
$.
Similarly, for $z$ we have $\rm{Stab}_{\Gamma}(z)= \{ \begin{bmatrix}
-b'z+d' & a'z+b'z^2-d'z \\
-b' & a'+b'z
\end{bmatrix}
:\ a'd'=1,\ a',b',d'\in\ZZ[i]\}
$.
If $v,z$ are in the same cusp, they have the same stabilizers modulo $\Gamma_0(I)$, thus for any choice $a,d,b$ as above there are $a',b',d'$ as above such that modulo $p$ the matrices above are equal.
We now show that this implies $v=z \pmod{p}$, say $v=z+tp$, thus for $\psi=[^{1\ tp}_{0\ \ 1}]\in \Gamma_0(I)$, $\psi(z)=v$ completing the proof.
Indeed, the above 4 equations give, modulo $p$, $b'\equiv b$ by bottom-left entry, substituting into top-left entry gives $d'\equiv d-b(v-z)$ and into bottom-right entry gives $a'\equiv a+b(v-z)$, finally substituting all these into the top-right entry gives $bv^2\equiv 2bv(v-z)+bz^2$. Assuming $z\neq v \pmod{p}$ gives $z+v\equiv 2v$, thus $z\equiv v$ after all.

(3) As $\Gamma$ acts transitively on the octahedra in the tiling and
$\Gamma \cap O$ acts transitively on the vertices of $T$, it is enough to show that if $T'$ is an octahedron in the tiling  which shares the cusps at $\frac{1}{2}+\frac{1}{2}i$ and $\infty$ with (the image of) $T$
in $\HH/\Gamma_0(I)$, then $T'=T$.
Note that there is an element $\psi\in PSL_2(F_q)$ in the intersection of the stabilizers of the above two cusps
mapping $T'$ to $T$. Thus, what we need to show is that all the elements in this intersection of stabilizers just rotate the other $4$ vertices of $T$ around the axis from $\frac{1}{2}+\frac{1}{2}i$ to $\infty$.

The stabilizer of the cusp at $\frac{1}{2}+\frac{1}{2}i$ is $\phi^{-1}\rm{Stab}(\infty)\phi$ for $\phi=[^{-1\ \ 0}_{1-i\ -1}]$,
thus consists of matrices of the form
$\begin{bmatrix}
a-b+ib & b \\
a-ia+2ib-d+id & b-ib+d
\end{bmatrix}
$
where
$ad=1$ and $a,b,d\in\ZZ[i]$, up to $\Gamma_0(I)$, while
$\rm{Stab}(\infty)$ consists of the matrices $\begin{bmatrix}
a' & b' \\
0 & d'
\end{bmatrix}$ where
$a'd'=1$ and $a',b',d'\in\ZZ[i]$, up to $\Gamma_0(I)$.

We compute the intersection of these stabilizers in $PSL_2(F_q)$, by equating the two types of matrices.
Modulo $p$, the bottom-left equation gives, for $q>2$,
$i(i+1)b\equiv d-a$.
Either $a=d=1$, or $a=i=-d$. In the first case (using $q>2$ again) we get $b\equiv 0$ and $\psi$ is the identity permutation. In the second case $b\equiv i-1$, substituting in top-left and bottom-right gives $\psi=[^{-i\  i-1}_{0\ \ i}]$ which permutes the vertices of (Euclidean) $T$ by $\pi$-rotation w.r.t. the axis through the vertices $\frac{1}{2}+\frac{1}{2}i$ and $\infty$.


(4) First, let $T'$ be
an octahedron in the tiling with opposite ends at cusps $a',b'$.
We need to show that there is no edge between $a'$ and $b'$.
Again, as the action of $\Gamma$ is transitive on pairs of antipodal ideal vertices of the octahedra, we can assume that $T'=T$ and $a'=\infty, b'=\frac{1}{2}+\frac{1}{2}i$ in $\HH/\Gamma_0(I)$.

If $a'b'$ were also an edge, then there would be a $\phi'\in \Gamma_0(I)$ for which $\phi'(b')\in \ZZ[i]$; equivalently there would be $\phi\in \Gamma_0(I)$ such that $\phi(0)=(\frac{1}{2}+\frac{1}{2}i)+z$ for some $z\in \ZZ[i]$. We show this is impossible when $q>5$:
as $I$ is prime we can write $I=(p)$.
Then $\phi=\begin{bmatrix}
ap+1 & bp \\
cp & dp+1
\end{bmatrix}$,
for some $a,b,c,d\in \ZZ[i]$,
and
$\phi(0)=\frac{bp}{dp+1}=(\frac{1}{2}+\frac{1}{2}i)+z$, yielding the equation
$2bp=(dp+1)(1+i+2z)$, so dividing by $1+i$ we get
$$(i+1)bp=(dp+1)(1+(1-i)z).$$
As $\det(\phi)=1$, $\gcd(bp,dp+1)=1$, and thus the displayed equation gives that $dp+1$ divides $1+i$. This implies that the Gaussian norm of $p$ is at most $5$, contradicting $q>5$.
\end{proof}
\textbf{From pseudomanifolds to manifolds.}
As mentioned, the cellulation from Theorem~\ref{lem:X} is a pseudomanifold where vertex links are tori, denote it by $X$, which we now modify to a polyhedral manifold $M=M(q)$. Each vertex of $X$  will be replaced by $3$ vertices in a way that will make the vertex links $2$-spheres.
Thus, $M$ will have $\Theta(q^{2})$ vertices and $\Theta(q^3)$ octahedra, to independently triangulate each by inserting any of its $3$ diagonals, so for $n=\frac{3(q^2-1)}{4}$ we construct $2^{\Theta(n^{3/2})}$ $n$-vertex triangulations of $M$.

At each cusp, cut along cycles representing the same meridian, to partition the torus into $3$ adjacent cylinders, each cylinder is closed on one boundary component and is open on the other boundary component. In each cylinder, identify all vertices (of the squares) to a single vertex, to get $3$ vertices $a,b,c$ per cusp, forming the boundary of a triangle. Then the link of $a$ is the polyhedral $2$-sphere formed from the cylinder corresponding to $a$ by coning the closed boundary component by say $c$ and compactifying the open boundary component by $b$; similarly we get that the links of $b$ and $c$ are $2$-spheres. 
The link of the edge $ab$ is a $1$-sphere corresponding to the meridian cut; and similarly for the edges $ac$ and $bc$.
We obtained a polyhedral $n$-vertex manifold $M$ as claimed, namely with $\Theta(n^{3/2})$ octahedra. (In fact we can triangulate the non-octahedron non-tetrahedron cells, without introducing new vertices, each in more than one way, but this will not change the asymptotics of the number of triangulations constructed).


\section{PBIBDs}\label{sec:PBIBD}
First, we recall some definitions.
An \emph{association scheme with $m$ classes} on a finite set $Y$ is a partition of $Y\times Y$ into $m+1$ nonempty parts $C_i$ (called \emph{associate classes}), $Y\times Y=\uplus_{0\leq i\leq m}C_i$, with $C_0=\{(y,y):\ y\in Y\}$, and such that (i) for each $0\leq j\leq m$ there is an integer $c_j$ such that for any $x\in Y$, $|\{y\in Y:\ (x,y)\in C_j\}|=c_j$, and (ii) for any $1\leq i,j,k\leq m$ there is an integer $p^k_{i,j}$ such that for any $(x,y)\in C_k$, $|\{z\in Y:\ (x,z)\in C_i\ \wedge \ (z,y)\in C_j\}|=p^k_{i,j}$.

A \emph{PBIBD(m)} with parameters $(v,b,r,k;\lambda_1,\ldots,\lambda_m)$ consists of a set $Y$ of size $v$ and a collection of $b$ subsets (called \emph{blocks}) of $Y$, each of size $k$, such that each $x\in Y$ belongs to exactly $r$ blocks; and there is an association scheme with $m+1$ classes on $Y$, $\uplus_{0\leq i\leq m}C_i$, such that for any two distinct $x,y\in Y$, if $(x,y)\in C_i$ then
$(x,y)$ is contained in exactly $\lambda_i$ of the blocks.

The case $m=1$ gives the classical $2$-designs, also known as BIBDs.
As we shall see, in our case, while $m=m(q)$ grows to infinity with $q$, all $\lambda_i$ but $\lambda_1$ and $\lambda_2$ vanish.

Interpreting the octahedra in our construction as blocks on the ground set of vertices, we notice the following.
\begin{remark}
The pseudomanifold $X=X_q$ constructed in Theorem~\ref{lem:X} (with $q>5$) gives a PBIBD($m$) with the following parameters:
blocks of size $k=6$, $v=(q^2-1)/4$ vertices, $b=q(q^2-1)/24$ blocks, each vertex is in exactly $r=q$ blocks, and each pair is in either
$\lambda_1=4$ blocks (for edges of $X$), or in $\lambda_2=1$ blocks (for diagonals in the octahedra of $X$),
or in $\lambda_i=0$ blocks (for the other pairs of vertices in $X$, $i\neq 1,2$). Lastly, $m=m(q)\geq \frac{q}{8}$.
\end{remark}
Indeed, identifying as usual the action of $G=PSL_2(F_q)$ on the vertices of $X$ (and the octahedra) with the action of $G$ on the left cosets $G/\rm{Stab}_G(\overline{\infty})$ (and the corresponding blocks\footnote{Here we used the fact that $q>5$, so no two octahedra have the same vertex set. The case $q=5$ is not interesting as a PBIBD; see the next remark for what we get when $q=5$.}), gives rise to a (Schurian) association scheme.
As $G$ acts transitively on the edges of $X$, and on the diagonals of $X$,
we get that $\lambda_1=4,\lambda_2=1$ and $\lambda_i=0$ for $2<i\leq m$.
Now $m=m(q)$ equals the number of vertex orbits under the action of $\rm{Stab}(\overline{\infty})$ minus 1, thus, omitting the orbit $\{\overline{\infty}\}$ and recalling that $|\rm{Stab}(\overline{\infty})|=2q$, we get $m \geq \lceil\frac{q^2-5}{4}/2q\rceil \geq \frac{q}{8}$.
The values of the parameters $k,v,b,r$ of the PBIBD above follow from Lemmas \ref{lem:count cusps and octahedra} and \ref{lem:combinatorics of tiling}.

\begin{remark}
In case $q=5$, there are $6$ vertices, $5$ octahedra, and $5$ squares per torus cusp -- so the graph is complete, namely $K_6$.
Identifying each octahedron with its $3$ diagonals, by Lemma \ref{lem:combinatorics of tiling} we obtain a \emph{$1$-factorization} of the edge set of $K_6$, namely a partition of the $15$ edges into $5$ triples where each triple is a perfect matching in $K_6$.
For more on $1$-factorizations on $K_6$ see e.g. \cite[Chapter 6]{Cameron-VanLint}.
\end{remark}

\section{Polyhedral surfaces}\label{sec:Surface}

We can get a sequence of regular polyhedral surfaces by a construction similar to the previous one.
These were constructed by McMullen \cite{McMullen:PLG91} and Wajima \cite{Wajima89}, by a different approach -- starting from the symmetry group, which is slightly bigger than $PSL_2(q)$, and without refereing to the tiling of the hyperbolic plane.

\textbf{The construction}: look on $PSL_2(\RR)$, the group of orientation preserving isometries of the hyperbolic plane $\HH^2$,
given by M\"{o}bius transformations on the upper half plane model.

Tile $\HH^2$ by ideal triangles as follows. In the half plane model, let $\Delta$ be the ideal triangle with vertices $0,1,\infty$, then its images $PSL_2(\ZZ)(\Delta)$ is the tiling of $\HH^2$ we speak of.
For an odd prime $q$ let
$\Gamma_0(q)$ be the kernel of the surjective map $PSL_2(\ZZ)\rightarrow PSL_2(\ZZ/q\ZZ)$, so it is torsion-free. Then $\HH^2/\Gamma_0(q)$ is a surface with cusps, each cusp looks like a $q$-cycle. We make it a compact surface $S$ by compactifying each cusp by adding one point to it.

The following lemma deals with the combinatorics of the surface $S$.
\begin{lemma}
Let $S=S_q$ for odd prime $q$ be as above. Then

1. $S$ is a simplicial complex, where at each vertex $q$ triangles meet.

2. $S$ consists of $v=\frac{q^2-1}{2}$ vertices,  $e=\frac{q(q^2-1)}{4}$ edges and  $t=\frac{q(q^2-1)}{6}$ triangles.

\end{lemma}
\begin{proof}
(1) First we make sure that the vertices of each triangle are in $3$ different cusps. As $PSL_2(\ZZ)$ acts transitively on the triangles of the tiling of $\HH^2$ we may assume the triangle is $\Delta$, and as $PSL_2(\ZZ)$ contains the $3$ rotations of $\Delta$ it is enough to check that $0$ and $\infty$ are in different cusps of $\HH^2/\Gamma_0(q)$. This we do by showing they have different stabilizers in $PSL_2(q)$.

The stabilizer in $PSL_2(\ZZ)$ of $\infty\in \HH^2$ consists of the distinct representative matrices $[^{1 \ x}_{0 \ 1}]$ for $x\in\ZZ$, so the stabilizer in $PSL_2(q)$ of $\infty\in \HH^2/\Gamma_0(q)$ is $\rm{Stab}(\infty)=\{[^{1 \ x}_{0 \ 1}]:\ x\in F_q\}$. Thus,
$\rm{Stab}(0)=[^{0 \  \  \ 1}_{-1 \ 0}]\rm{Stab}(\infty)[^{0 \ -1}_{1 \  \ 0}]
=\{[^{1 \ 0}_{y \ 1}]:\ y\in F_q\}$.

To complete the proof that $S$ is a simplicial complex we need to check there are no multiple edges between any pair of cusps.
Again, by the transitive action of $PSL_2(\ZZ)$ on the cusps, and on edges of the tiling in $\HH^2$, we only need to show that for two edges from $\infty$, with the other ends at $v,z\in\ZZ$, one has $v\neq z\pmod{q}$. Again, we compare stabilizers.

$\rm{Stab}(v)=[^{v \ -1}_{1 \ \ 0}]\rm{Stab}(\infty)[^{0 \ \ \ 1}_{-1 \ v}]
=\{[^{1-vx \ xv^2}_{-x \ \ 1+vx}]:\ x\in F_q\}$,
and similarly
$\rm{Stab}(z)=\{[^{1-zx' \ x'z^2}_{-x' \ \ 1+zx'}]:\ x'\in F_q\}$.
If these stabilizers are equal, then the bottom-left entry of the matrices gives, modulo $q$, $x\equiv x'$ so for $x$ nonzero by the top-left entries we get $v\equiv z$.

Mod $\Gamma_0(q)$, the horocycle at $\infty$ looks like a $q$-cycle, each of its edges corresponds to a triangle with a vertex at $\infty$. By transitivity of the action of $PSL_2(q)$ on the cusps, all vertex links look the same.

(2)
$|PSL_2(q)|=\frac{q(q^2-1)}{2}$ and as we saw the stabilizer of a cusp has size $q$, so we get $v=\frac{q^2-1}{2}$.
Each cusp looks like a $q$-cycle, so the graph is $q$ regular and we get $e=qv/2$.
As $\Gamma$ contains $3$ of the symmetries of $\Delta$ (by rotations), $t=|PSL_2(q)|/3$.


\end{proof}

We now verify that $S$ we constructed is the same as the polyhedral surfaces described by McMullen in \cite[Theorem 2]{McMullen:PLG91}. We use the notation as appears there, and do not repeat his construction    .

\begin{lemma}\label{lem:McMullen}
For $q>2$ a prime, the polyhedral surface $S=S_q$ constructed above is combinatorially isomorphic to $\Pi=\Pi_q$ constructed in \cite[Theorem 2]{McMullen:PLG91}.
\end{lemma}
\begin{proof}
Let $V(S)$ be the vertex set of $S_q$ and $V(\Pi)=(F_q^2-\{0\})/(x\sim -x:\ 0\neq x\in F_q^2)$ the vertex set of $\Pi_q$, and define $f:V(S)\rightarrow V(\Pi)$ as follows. Let $f(\infty)=[0,1]$. Note that the stabilizers in $PSL_2(q)$ of both $\infty$ and $[0,1]$ are the same, namely the matrices $\{[^{1\ x}_{0\ 1}]:\ x\in F_q\}$. Thus, as $PSL_2(q)$ acts transitively on both $V(S)$ and $V(\Pi)$, the map $f(g(\infty)))=g([0,1])$ for $g\in PSL_2(q)$ is a well defined bijection from $V(S)$ to $V(\Pi)$. Fix an orientation on $S$ and $\Pi$, induced by a fixed orientation on their common universal cover, namely $\HH^2$ (in case $q=5$ the universal cover is the $2$-sphere, and both $S_5$ and $\Pi_5$ are the icosahedron).
To see that $f$ induces a combinatorial isomorphism from $S_q$ to $\Pi_q$, first note that the links of both $\infty$ and $[0,1]$ are $q$-cycles. To finish, it is enough to notice that $PSL_2(q)$ acts transitively on \emph{oriented flags} $(v,e,t)$, namely a vertex $v$ in an edge $e$ in a triangle $t$, such that the orientation from $v$ to $e$ to $t$ agrees with the orientation of the surface.
\end{proof}

\section{Concluding remarks}
Hopefully this note will be useful for future research.
Natural questions that arise are to analyze similar constructions where the hyperbolic space / plane is tiled by different tiles, and where quotients by different discrete subgroups are taken.
In this note we have not explored these directions.
For related works from the point of view of abstract polytopes see \cite{Schulte-Egon-Weiss} and references therein, and for constructions of $3$-designs from $PSL_2(q)$ see \cite{3-designs} and references therein.
Further treatment of the PBIBDs constructed was also left out.

The $1$-skeleton of the objects constructed are $q$-regular graphs of order $\Theta(q^2)$. Their expansion properties are another direction not explored here.

\section*{Acknowledgments}
I would like to thank Nir Avni, Misha Klin, Ron Livne, Jacek \'{S}wi\c{a}tkowski and especially Danny Kalmanovich
for helpful discussions, and the anonymous referee for helpful suggestions on the presentation.

\bibliographystyle{plain}
\bibliography{pbiblio}

\end{document}